\documentclass[12pt,twoside]{amsart}
\usepackage{latexsym,amsfonts,amsmath,amssymb,soul}
\usepackage{enumerate,soul}
\usepackage[titletoc]{appendix}
\usepackage[usenames, dvipsnames]{color}
\usepackage{tikz}
\numberwithin{equation}{section}

\newtheorem{Th}{Theorem}[section]

\newtheorem{Lemma}[Th]{Lemma}
\newtheorem{Def}[Th]{Definition}
\newtheorem{Prop}[Th]{Proposition}

\catcode`\@=11

\renewcommand{\section}%
   {\setcounter{equation}{0}\@startsection {section}{1}{\z@}{-3.5ex plus -1ex
  minus -.2ex}{2.3ex plus .2ex}{\Large\bf}}

\def\Im{\mathop{\rm Im}\nolimits}
\def\Id{\mathop{\rm Id}\nolimits}

\def\esssup{\mathop{\rm ess\,sup}}
\def\ds{\displaystyle}
\def\R{\mathbb R}
\def\C{\mathbb C}
\def\N{\mathbb N}

\def\Z{\mathbb Z}

\newcommand{\F}{\mathcal{F}}

\newcommand{\Sch}{\mathcal{S}}

\newcommand{\afrac}[2]{\genfrac{}{}{0pt}{1}{#1}{#2}}

\newcommand{\beqsn}{\arraycolsep1.5pt\begin{eqnarray*}}
\newcommand{\eeqsn}{\end{eqnarray*}\arraycolsep5pt}
\newcommand{\beqs}{\arraycolsep1.5pt\begin{eqnarray}}
\newcommand{\eeqs}{\end{eqnarray}\arraycolsep5pt}

\title{Nuclearity of rapidly decreasing ultradifferentiable functions and time-frequency analysis}

\author[Boiti]{Chiara Boiti}
\address{
Dipartimento di Matematica e Informatica \\Universit\`a di Ferrara\\
Via Ma\-chia\-vel\-li n.~30\\
I-44121 Ferrara\\
Italy}
\email{chiara.boiti@unife.it}

\author[Jornet]{David Jornet}
\address{
Instituto Universitario de Matem\'atica Pura y Aplicada IUMPA\\
Universitat Po\-li\-t\`ecni\-ca de Val\`encia\\
Camino de Vera, s/n\\
E-46071 Valencia\\
Spain}
\email{djornet@mat.upv.es}

\author[Oliaro]{Alessandro Oliaro}
\address{Dipartimento di Matematica\\ Universit\`a di Torino\\
 Via Carlo Alberto n.~10\\ I-10123 Torino\\ Italy}
 \email{alessandro.oliaro@unito.it}

\author[Schindl]{Gerhard Schindl}
\address{Fakult\"at f\"ur Mathematik\\ Universit\"at Wien\\
Oskar-Morgenstern-Platz n.~1\\ A-1090 Wien\\ Austria}
 \email{gerhard.schindl@univie.ac.at}

\begin{document}


\begin{abstract}
We use techniques from time-frequency analysis to show that the space $\Sch_\omega$ of rapidly decreasing $\omega$-ultradifferentiable functions is nuclear for every weight function  $\omega(t)=o(t)$ as $t$ tends to infinity. Moreover, we prove that,
for a sequence $(M_p)_p$ satisfying the classical condition $(M1)$ of Komatsu, the
space of Beurling type $\Sch_{(M_p)}$ when defined with $L^{2}$~norms is nuclear exactly when condition $(M2)'$ of Komatsu holds.
\end{abstract}

\maketitle

%
\markboth{\sc  Nuclearity of rapidly decreasing ultradifferentiable functions...}
 {\sc C.~Boiti, D.~Jornet, A.~Oliaro and G.~Schindl}

\section{Introduction}

One of the main properties of a nuclear space is that the Schwartz kernel theorem holds, which gives, for instance, a different representation of a continuous and linear pseudodifferential operator as an integral operator in terms of its kernel. This is very useful for the study of the propagation of singularities or the behaviour  of wave front sets of pseudodifferential  operators. See, for example, \cite{AJ,BJO-Gabor,fgj1,fgj2,R,RW} and the references therein. 

In fact, in \cite{BJO-Gabor} the first three authors of the present work  imposed the following condition on the weight function: there is $H>1$ such that for every $t>0$, $2\omega(t)\le \omega(Ht)+H$ (see \cite[Corollary 16\,(3)]{BMM}), to have that the space $\Sch_{\omega}(\R^{d})$ is nuclear (see \cite{BJO-Rodino}). Hence they could analyse the kernel of some pseudodifferential operators \cite[Section 4]{BJO-Gabor}. In the present paper, we complete the study begun in \cite{BJO-Rodino} and prove that $\Sch_{\omega}(\R^{d})$ is nuclear for every weight function $\omega(t)=o(t)$ as $t$ tends to infinity (see Definition~\ref{defomega}).  Hence, now the powers of the logarithm $\omega(t)=\log^{\beta}(1+t)$, $\beta>0$, are allowed as weight functions and, in particular, we recover a known result for the weight $\omega(t)=\log(1+t),$ namely, that the Schwartz class $\Sch(\R^{d})$ is a nuclear space.

  To see that $\Sch_{\omega}(\R^{d})$ is nuclear we establish an isomorphism, which is new in the literature, with some Fr\'echet sequence  space.~We use expansions in terms of Gabor frames, that are a fundamental tool in time-frequency analysis.~This is motivated by the rapid decay of the Gabor coefficients of a function in $\Sch_{\omega}(\R^{d})$ when 
  $\omega$ is a subadditive function, as we showed in \cite{BJO-Gabor}. More precisely, 
    we proved  that $u\in\Sch_\omega(\R^{d})$
 if and only if 
 \beqsn
 \sup_{\sigma\in\alpha_0\Z^d\times\beta_0\Z^d}e^{\lambda\omega(\sigma)}
 |V_\varphi u(\sigma)|<+\infty,\qquad\mbox{ for all }\lambda>0,
 \eeqsn
 where $\alpha_0,\beta_0>0$ are sufficiently small so that $\{\Pi(\sigma)\varphi\}_{\sigma\in
 \alpha_0\Z^d\times\beta_0\Z^d}$ is a Gabor frame in $L^2(\R^d)$ for a fixed
 window function $\varphi\in\Sch_\omega(\R^{d})$,  $V_\varphi u$ is the short-time
 Fourier transform of $u$ and $\Pi(\sigma)$ is the time-frequency
 shift defined as $\Pi(\sigma)\varphi(y)=e^{i\langle y,\beta_0 n\rangle} \varphi(y-\alpha_0k)$,
 for $\sigma=(\alpha_0k,\beta_0n)$. The usual properties of modulation spaces in \cite{BJO-Gabor} hold only when the weight function $\omega$ is subadditive. However, the expansion in terms of Gabor frames is still possible when the weight is non subadditive and satisfies $\omega(t)=o(t)$ as $t$ tends to infinity. In fact, we prove here that 
$\Sch_\omega(\R^{d})$ is isomorphic to a topological subspace of the sequence space
\beqs
\label{Lambdaomega}
\tilde{\Lambda}_\omega:=\{c=(c_\sigma)_{\sigma\in\alpha_0\Z^d\times\beta_0\Z^d}:
\|c\|_k:=\sup_{\sigma\in\alpha_0\Z^d\times\beta_0\Z^d}|c_\sigma|e^{k\omega(\sigma)}
<+\infty,
\ \forall \, k\in\N\}. 
\eeqs
The isomorphism is defined in \eqref{iso} by the restriction on its image of the \emph{analysis} operator, that maps $u\in\Sch_\omega(\R^{d})$ to its Gabor coefficients
$\{V_\varphi u(\sigma)\}_{\sigma\in \alpha_0\Z^d\times\beta_0\Z^d}$. As a consequence,  $\Sch_{\omega}(\R^{d})$ is nuclear by an application of Grothendieck-Pietsch criterion to the  space
$\tilde{\Lambda}_\omega$ (Proposition~\ref{cor1}). This isomorphism is not the only one existing in the literature, and it should be compared with the one given by Aubry~\cite{A}, \emph{only} for the one-variable case, obtaining that $\Sch_{\omega}(\R)$ is isomorphic to the different sequence space 
\beqsn
\Lambda_{\omega}:=\Big\{(c_k)_{k\in\N_0}:\
\sup_{k\in \N_0}|c_k|e^{j\omega(k^{1/2})}<+\infty,\ \forall \,j\in\N_0\Big\}.
\eeqsn
 Aubry uses expansions in terms of the Hermite functions, as Langenbruch \cite{L} did previously for spaces defined by sequences in the sense of Komatsu. 
 
Finally, in the last section of the paper, and without using techniques from time frequency-analysis, we characterize when the Beurling space of ultradifferentiable functions  $\Sch_{(M_{p})}(\R^{d})$ (see formula \eqref{SMp} for the definition) in the sense of Komatsu is nuclear. We can give such a characterization when the space is defined by $L^{2}$ norms. We explain and motivate a little bit this result.  Pilipovi\'c, Prangoski and Vindas~\cite{PPV} showed that 
\beqs
\label{SMp-infty}
\qquad
S_{(M_p)}(\R^d):=\Big\{f\in C^\infty(\R^d): \
\sup_{\alpha,\beta\in\N_0^d}\sup_{x\in\R^d}\frac{j^{|\alpha+\beta|}}{M_{|\alpha+\beta|}}
\|x^\alpha \partial^\beta f(x)\|_\infty<+\infty,\ \forall\, j\in\N \Big\},
\eeqs is nuclear when $(M_p)_p$ satisfies the standard conditions $(M1)$ (defined below in formula \eqref{M1}) and $(M2)$ (that we do not define here because it is not used), which is stronger than $(M2)'$, defined below in formula \eqref{M2}. On the other hand, using the isomorphism of \cite{L}, we proved in \cite{BJO-Rodino} that the space $S_{(M_p)}(\R^d)$ is nuclear if $(M_p)_p$ satisfies that
there is $H>0$ such that for any $C >0$ there is $B >0$ with 
\beqs
\label{12L}
s^{s/2} M_p\leq BC^sH^{s+p}
M_{s+p},\qquad\forall s,p\in\N_0
\eeqs
and $(M2)'$ {\em (stability under differential operators)}:
\beqs\label{M2}
\exists A,H>0:\quad M_{p+1}\leq AH^p M_p,
\qquad\forall p\in\N_0.
\eeqs
The condition \eqref{12L} is quite natural and not restrictive at all and it is used by Langenbruch~\cite{L} to show that the Hermite functions are elements of $S_{(M_p)}(\R^d)$. Moreover, Langenbruch also proves in \cite[Remark 2.1]{L} that under these two conditions \eqref{12L} and \eqref{M2}, $S_{(M_p)}(\R^d)=\Sch_{(M_p)}(\R^d)$.  If we do not assume \eqref{M2} and consider only $\Sch_{(M_p)}(\R^d)$  (the space defined with $L^{2}$~norms), after a careful reading of the proofs of some results of \cite{L} in the Beurling case and the use of  techniques of Petzsche~\cite{P}, we are able to prove here that under the additional conditions $(M1)$ {\em (logarithmic convexity)}:
\beqs\label{M1}
M_p^2\leq M_{p-1}M_{p+1},\qquad \forall\, p\in\N,
\eeqs
and that 
$M_p^{1/p}\to+\infty$ as $p\to+\infty$,  $\Sch_{(M_p)}(\R^d)$ is nuclear if and only if $(M_{p})_{p}$ satisfies $(M2)'$.

The paper is organized as follows. In Section~\ref{sec2}, we show that Gabor frames have a stable behaviour with the only assumption $\omega(t)=o(t)$ as $t$ tends to infinity on the weight function. Indeed, we see that the \emph{analysis} and \emph{synthesis} operators are well defined and continuous in the suitable spaces (Propositions \ref{prop1Ale} and \ref{prop2Ale}), defining an isomorphism between $\Sch_\omega(\R^{d})$ and a subspace of $\tilde{\Lambda}_\omega$.  
In Section~\ref{sec3} we recover for this setting some known properties of K\"othe echelon spaces to see that the sequence space $\tilde\Lambda_\omega$ is  nuclear. And, finally, in Section~\ref{sec4} we characterize the nuclearity of $\Sch_{(M_p)}(\R^d)$.

\section{Gabor frame operators in $\Sch_\omega(\R^{d})$}
\label{sec2}

Let us condider weight functions of the form:
\begin{Def}
\label{defomega}
A {\em weight function}  is a continuous increasing function
$\omega:\ [0,+\infty)\to[0,+\infty)$ satisfying the following properties:
\begin{itemize}
\vspace{1mm}
\item[$(\alpha)$]
there is $L\geq 1\mbox{ such that }\omega(2t)\leq L(\omega(t)+1),$ for each $t\geq0;$
\vspace{1mm}
\item[$(\beta)$]
$
\omega(t)=o(t)\quad\mbox{as}\ t\to+\infty;
$
\vspace{1mm}
\item[$(\gamma)$]
there are $a\in\R, b>0\mbox{ such that }\omega(t)\geq a+b\log(1+t),\ t\geq0;
$
\vspace{1mm}
\item[$(\delta)$]
the map $t\mapsto 
\varphi_\omega(t):=\omega(e^t)\ \mbox{is convex}.
$
\end{itemize}
\vspace{1mm}
For $\zeta\in\C^d$, we put $\omega(\zeta):=\omega(|\zeta|)$, where $|\zeta|$ denotes the Euclidean norm of $\zeta$.
\end{Def}

Note that condition $(\alpha)$ implies
\beqs
\label{L}
\omega(t_1+t_2)\leq L(\omega(t_1)+\omega(t_2)+1),\qquad t_1,t_2\geq0.
\eeqs

We denote by $\varphi^*_\omega$ the {\em Young conjugate} of $\varphi_\omega$,
defined by
\beqsn
\varphi^*_\omega(s):=\sup_{t\geq0}\{ts-\varphi_\omega(t)\}.
\eeqsn
We recall that $\varphi^*_\omega$ is an increasing and convex function satisfying
$(\varphi^{*}_\omega)^{*}=\varphi_\omega$ (see \cite{hor-cvx}).
Moreover $\varphi^*_\omega(s)/s$ is increasing. For a collection of further well-known properties of $\varphi^*_\omega$ we refer, for instance, to \cite[Lemma~2.3]{BJO-PW}.

We consider the following notation for the Fourier transform of $u\in L^1(\R^d)$:
\beqsn
\F(u)(\xi)=\hat{u}(\xi):=\int_{\R^d}u(x)e^{-i\langle x,\xi\rangle}dx,
\qquad\xi\in\R^d,
\eeqsn
with standard extensions to more general spaces of functions or distributions.
We recover from \cite{Bj} the following
\begin{Def}
\label{Somega}
The space $\Sch_\omega(\R^d)$ is the set of all $u\in L^1(\R^d)$ such that
$u,\hat u\in C^\infty(\R^d)$ and for each $\lambda>0$ and each $\alpha\in\N_0^d$
we have
\beqsn
\sup_{x\in\R^d}e^{\lambda\omega(x)}|\partial^\alpha u(x)|<+\infty
\quad\mbox{and}\quad
\sup_{\xi\in\R^d}e^{\lambda\omega(\xi)}|\partial^\alpha \hat{u}(\xi)|<+\infty.
\eeqsn
The corresponding strong dual of ultradistributions will be denoted by
$\Sch'_\omega(\R^d)$.
\end{Def}

We denote by $T_x$, $M_\xi$ and $\Pi(z)$, respectively, the {\em translation}, the
{\em modulation} and the {\em phase-space shift} operators, defined by
\beqsn
&&T_x f(y)=f(y-x),\quad M_\xi f(y)=e^{i\langle y,\xi\rangle}f(y)\\
&& \Pi(z)f(y)=M_\xi T_xf(y)=e^{i\langle y,\xi\rangle}f(y-x)
\eeqsn
for $x,y,\xi\in\R^d$ and $z=(x,\xi)$.

For a window function $\varphi\in\Sch_\omega(\R^d)\setminus\{0\}$ the {\em short-time
Fourier transform} (briefly STFT) of $f\in\Sch'_\omega(\R^d)$ is defined, for
$z=(x,\xi)\in\R^{2d}$, by
\beqs
\label{3}
V_\varphi f(z):=&&\langle f,\Pi(z)\varphi\rangle\\
\label{4}
=&&\int_{\R^d}f(y)\overline{\varphi(y-x)}e^{-i\langle y,\xi\rangle}dy,
\eeqs
where the brackets $\langle\cdot,\cdot\rangle$ in \eqref{3} and the (formal) integral in \eqref{4} denote
the conjugate linear action of $\Sch'_\omega$ on $\Sch_\omega$,
consistent with the inner product $\langle\cdot,\cdot\rangle_{L^2}$.

By condition $(\gamma)$ of Definition~\ref{defomega} it is easy to deduce  that $\Sch_\omega(\R^d)\subset\Sch(\R^d)$. Hence, $\Sch_\omega(\R^d)$ can be equivalently defined as the set of all $u\in\Sch(\R^d)$ that satisfy the conditions of Definition~\ref{Somega}.
The Fourier transform $\mathcal{F}:\Sch_\omega(\R^d)\to \Sch_\omega(\R^d)$ is a continuous automorphism, that can be extended in the usual way to $\Sch_\omega^\prime(\R^d)$ and, moreover, the space $\Sch_\omega(\R^d)$ is an algebra under multiplication and convolution. On the other hand, for $u,\psi\in\Sch_\omega(\R^d)$ we have $V_\psi u\in \Sch_\omega(\R^{2d})$. Moreover, for $u\in\Sch'_\omega(\R^d)$ the short-time Fourier transform is well defined and belongs to
$\Sch_\omega^\prime(\R^{2d})$. We refer to \cite{Bj,GZ,BJO-Gabor}
for subadditive weights, and to \cite{BJO-Wigner,BJO-PW} for
non-subadditive weights; in particular, all results of \cite[Section~2]{BJO-Gabor} are valid in the non-subadditive case also.

We shall need the following theorem from \cite{BJO-PW}:
\begin{Th}
\label{propSomega}
Given a function $u\in\Sch(\R^d)$ and $1\leq p,q\leq+\infty$, we have that $u\in\Sch_\omega(\R^d)$ if and only
if one of the following equivalent conditions is satisfied:
\begin{itemize}
\item[$(a)$]
\begin{itemize}
\item[i)]
$\forall\lambda>0,\ \alpha\in\N^d_0\ \exists C_{\alpha,\lambda}>0$ s.t.
$\ds\|e^{\lambda\omega(x)} \partial^\alpha u(x)\|_{L^p}\leq C_{\alpha,\lambda}\,$,
\item[ii)]
$\forall\lambda>0,\ \alpha\in\N^d_0\ \exists C_{\alpha,\lambda}>0$ s.t.
$\ds\|e^{\lambda\omega(\xi)} \partial^\alpha\hat{u}(\xi)\|_{L^q}\leq C_{\alpha,\lambda}\,$;
\end{itemize}
\item[$(b)$]
\begin{itemize}
\item[i)]
$\forall\lambda>0,\ \alpha\in\N^d_0\ \exists C_{\alpha,\lambda}>0$ s.t.
$\ds\|e^{\lambda\omega(x)} x^\alpha u(x)\|_{L^p}\leq C_{\alpha,\lambda}\,$,
\item[ii)]
$\forall\lambda>0,\ \alpha\in\N^d_0\ \exists C_{\alpha,\lambda}>0$ s.t.
$\ds\|e^{\lambda\omega(\xi)} \xi^\alpha\hat{u}(\xi)\|_{L^q}\leq C_{\alpha,\lambda}\,$;
\end{itemize}
\item[$(c)$]
\begin{itemize}
\item[i)]
$\forall\lambda>0\ \exists C_{\lambda}>0$ s.t.
$\ds\|e^{\lambda\omega(x)} u(x)\|_{L^p}\leq C_{\lambda}\,$,
\item[ii)]
$\forall\lambda>0\ \exists C_{\lambda}>0$ s.t.
$\ds\|e^{\lambda\omega(\xi)}\hat{u}(\xi)\|_{L^q}\leq C_{\lambda}\,$;
\end{itemize}
\item[$(d)$]
\begin{itemize}
\item[i)]
$\forall\lambda>0,\beta\in\N_0^d\ \exists C_{\beta,\lambda}>0$ s.t.
$\ds\sup_{\alpha\in\N_0^d}\|x^\beta \partial^\alpha u(x)\|_{L^p}e^{-\lambda
\varphi^*_{\omega}\left(\frac{|\alpha|}{\lambda}\right)}\leq C_{\beta,\lambda}\,$,
\item[ii)]
$\forall\mu>0,\alpha\in\N_0^d\ \exists C_{\alpha,\mu}>0$ s.t.
$\ds\sup_{\beta\in\N_0^d}\|x^\beta \partial^\alpha u(x)\|_{L^q}e^{-\mu
\varphi^*_{\omega}\left(\frac{|\beta|}{\mu}\right)}\leq C_{\alpha,\mu}\,$;
\end{itemize}
\item[$(e)$]
$\forall\mu,\lambda>0\ \exists C_{\mu,\lambda}>0$ s.t.
$\ds\sup_{\alpha,\beta\in\N_0^d}\|x^\beta \partial^\alpha u(x)\|_{L^p}
e^{-\lambda\varphi^*_{\omega}\left(\frac{|\alpha|}{\lambda}\right)}e^{-\mu
\varphi^*_{\omega}\left(\frac{|\beta|}{\mu}\right)}\leq C_{\mu,\lambda}\,$;
\item[$(f)$]
$\forall\lambda>0\ \exists C_{\lambda}>0$ s.t.
$\ds\sup_{\alpha,\beta\in\N_0^d}\|x^\beta \partial^\alpha u(x)\|_{L^p}e^{-\lambda
\varphi^*_{\omega}\left(\frac{|\alpha+\beta|}{\lambda}\right)}\leq C_{\lambda}\,$;
\item[$(g)$]
$\forall\mu,\lambda>0\ \exists C_{\mu,\lambda}>0$ s.t.
$\ds\sup_{\alpha\in\N_0^d}\|e^{\mu\omega(x)} \partial^\alpha u(x)\|_{L^p}
e^{-\lambda\varphi^*_{\omega}\left(\frac{|\alpha|}{\lambda}\right)}\leq C_{\mu,\lambda}\,$;
\item[$(h)$]
Given $\psi\in\Sch_\omega(\R^d)\setminus\{0\}$, $\forall\lambda>0\ \exists C_{\lambda}>0$ s.t.
$\ds\|V_\psi u(z) e^{\lambda\omega(z)}\|_{L^{p,q}}\leq C_\lambda$.
\end{itemize}
\end{Th}

Let us set, for $\lambda\in\R\setminus\{0\}$,
\beqsn
m_\lambda(z)=e^{\lambda\omega(z)},\qquad z\in\R^{2d},
\eeqsn
and consider the weighted $L^{p,q}$ spaces
\beqsn
L^{p,q}_{m_\lambda}(\R^{2d}):=\Big\{&&\!\!\!F\ \mbox{measurable on $\R^{2d}$ such that}\\
&&\|F\|_{L^{p,q}_{m_\lambda}}:=\Big(\int_{\R^{d}}\Big(\int_{\R^{d}}|F(x,\xi)|^pm_\lambda (x,\xi)^pdx\Big)^{q/p}
d\xi\Big)^{1/q}<+\infty\Big\},
\eeqsn
for $1\leq p,q<+\infty$, and
\beqsn
L^{\infty,q}_{m_\lambda}(\R^{2d}):=\Big\{&&\!\!\!F\ \mbox{measurable on $\R^{2d}$ such that}\\
&&\|F\|_{L^{\infty,q}_{m_\lambda}}:=\Big(\int_{\R^{d}}\big(\esssup_{x\in\R^d}|F(x,\xi)|
m_\lambda(x,\xi)\big)^{q}
d\xi\Big)^{1/q}<+\infty\Big\},\\
L^{p,\infty}_{m_\lambda}(\R^{2d}):=\Big\{&&\!\!\!F\ \mbox{measurable on $\R^{2d}$ such that}\\
&&\|F\|_{L^{p,\infty}_{m_\lambda}}:=\esssup_{\xi\in\R^d}
\Big(\int_{\R^{d}}|F(x,\xi)|^pm_\lambda(x,\xi)^p dx\Big)^{1/p}
<+\infty\Big\},
\eeqsn
for $1\leq p,q\leq+\infty$ with $p=+\infty$ or $q=+\infty$ respectively. 
If $p=q$ we write $L^p_{m_\lambda}(\R^d)=L^{p,q}_{m_\lambda}(\R^d)$.

%

Here we consider generic weight functions $\omega$ satisfying
$(\alpha)$ of Definition~\ref{defomega} (weaker than subadditivity).
In this case modulation spaces lack several properties. Hence we prove directly
some results on Gabor frames in $\Sch_\omega(\R^d)$ without using modulation spaces.
%
If $\omega$ is subadditive we know, by Theorem~4.2 of \cite{GZ},
that for any fixed $\varphi\in
\Sch_\omega(\R^d)$ its dual window $\psi$, in the sense of the theory of Gabor frames (see Gr\"ochenig~\cite{G}), belongs to $\Sch_\omega(\R^d)$ (see also
 \cite[Thm. 4.2]{GL}  and \cite{J}).
 In our case we will fix
$\varphi_0(x):=e^{-|x|^2}$  the Gaussian function,
$\alpha_0,\beta_0>0$ such that $\{\Pi(\sigma)\varphi_0\}_{\sigma\in\alpha_0\Z^d\times\beta_0\Z^d}$ is a Gabor frame
for $L^2(\R^d)$, and then prove that the canonical dual window $\psi_0$ of
$\varphi_0$ is in $\Sch_\omega(\R^d)$. To this aim we start by the following

\begin{Lemma}
\label{lemmasubad}
Let $\omega$ be a weight function. There exists then a subadditive weight function
$\sigma$ such that $\omega(t)=o(\sigma(t))$ as $t\to+\infty$.
\end{Lemma}

\begin{proof}
Let us consider $\omega_0(t)=\max\{0,t-1\}$. This is a continuous increasing
function $\omega_0:\ [0,+\infty)\to[0,+\infty)$ that satisfies $(\alpha)$, $(\gamma)$
and $(\delta)$ of Definition~\ref{defomega} and moreover $\omega_0|_{[0,1]}\equiv0$ and
$\omega_0$ is concave on $[1+\infty)$.

Then, by Lemma~1.7 and Remark~1.8(1) of \cite{BMT}, there exists a weight function
$\lambda$ satisfying $(\alpha)$, $(\gamma)$ and $(\delta)$ and such that
$\lambda|_{[0,1]}\equiv0$, $\lambda$ concave on $[1,+\infty)$ and
\beqsn
&&\omega(t)=o(\lambda(t)),\qquad\mbox{as}\ t\to+\infty,\\
&&\lambda(t)=o(\omega_0(t))=o(t),\qquad\mbox{as}\ t\to+\infty.
\eeqsn
Since $\lambda(t+1)$ is concave on $[0,+\infty)$ with
$\lambda(1)=0$, we have that 
$\sigma(t):=\lambda(t)+\lambda(2)$ is the required weight function.
%
%
%
%
\end{proof}

\begin{Prop}
\label{Gaussiana}
Let $\varphi_0(x)=e^{- |x|^2}$ be the Gaussian function and let
$\psi_0$ be a dual window of $\varphi_0$. Then $\varphi_0,\psi_0\in\Sch_\omega(\R^d)$
for every weight function $\omega$.
\end{Prop}

\begin{proof}
Let $\omega$ be a weight function as in Definition~\ref{defomega}.
By Lemma~\ref{lemmasubad} there exists a subadditive weight function $\sigma$
such that $\omega(t)=o(\sigma(t))$ as $t\to+\infty$. Then
$\Sch_\sigma(\R^d)\subseteq\Sch_\omega(\R^d)$.

Clearly $\varphi_0\in\Sch_\sigma(\R^d)\subseteq\Sch_\omega(\R^d)$ by
condition $(\beta)$. Since $\sigma$ is subadditive, by \cite[Thm. 4.2]{GZ},
its dual window $\psi_0\in\Sch_\sigma(\R^d)\subseteq\Sch_\omega(\R^d)$ and the proof is complete.
\end{proof}

We fix, once and for all, $\varphi_0(x)=e^{- |x|^2}$,
$\alpha_0,\beta_0>0$ such that $\{\Pi(\sigma)\varphi_0\}_{\sigma\in\alpha_0\Z^d\times\beta_0\Z^d}$ is a Gabor frame for $L^2$ and $\psi_0$ the canonical dual window
of $\varphi_0$ (see \cite[Section~7.3]{G}). For the lattice $\Lambda:=\alpha_0\Z^d\times\beta_0\Z^d$, we consider
the {\em analysis operator} $C_{\varphi_0}$ acting on a function $f\in L^2(\R^d)$
\beqsn
C_{\varphi_0} f:=\langle f,\Pi(\sigma)\varphi_0\rangle,\qquad\sigma\in\Lambda,
\eeqsn
and the {\em synthesis operator} $D_{\psi_0}$ acting on a sequence $c=(c_{k,n})_{k,n\in\Z^d}$
\beqsn
D_{\psi_0} c=\sum_{k,n\in\Z^d}c_{k,n}\Pi(\alpha_0k,\beta_0n)\psi_0.
\eeqsn
It is well known (see, for instance, \cite{G}) that
\beqsn
D_{\psi_0} C_{\varphi_0}=\Id,\qquad\mbox{the identity on}\ L^2(\R^d),
\eeqsn
since $\psi_0$ is the canonical dual window of $\varphi_0$, and then
\beqs
\label{add1}
D_{\psi_0} C_{\varphi_0}=\Id,\qquad\mbox{on}\ \Sch_\omega(\R^d)\subset L^2(\R^d).
\eeqs
Later on we shall explain more precisely this identity on $\Sch_\omega(\R^d)$.

We denote by $\ell^{p,q}_{m_\lambda}$, for $1\leq p,q\leq+\infty$ and $\lambda\in\R\setminus\{0\}$, the space of all
sequences $a=(a_{kn})_{k,n\in\Z^d}$, with $a_{kn}\in\C$ for every $k,n\in\Z^d$, such that
\beqsn
\|a\|_{\ell^{p,q}_{m_\lambda}}:=\left(\sum_{n\in\Z^d}\left(\sum_{k\in\Z^d}
|a_{kn}|^pm_\lambda(k,n)^p\right)^{q/p}\right)^{1/q}<+\infty,
\eeqsn
if $1\leq p,q<+\infty$,
\beqsn
&&\|a\|_{\ell^{\infty,q}_{m_\lambda}}:=\left(\sum_{n\in\Z^d}\left(\sup_{k\in\Z^d}
|a_{kn}|m_\lambda(k,n)\right)^{q}\right)^{1/q}<+\infty,\\
&&\|a\|_{\ell^{p,\infty}_{m_\lambda}}:=\sup_{n\in\Z^d}\left(\sum_{k\in\Z^d}
|a_{kn}|^pm_\lambda(k,n)^p\right)^{1/p}<+\infty,
\eeqsn
for $1\leq p,q\leq+\infty$ with $p=+\infty$ or $q=+\infty$ respectively.

Then we say that a measurable function $F$ on $\R^{2d}$ belongs to the
{\em amalgam space} $W(L^{p,q}_{m_\lambda})$ for the sequence
\beqsn
a_{kn}:=\esssup_{(x,\xi)\in[0,1]^{2d}}|F(k+x,n+\xi)|
=\|F\cdot T_{(k,n)}\chi_Q\|_{L^\infty},
\eeqsn
where $\chi_Q$ is the characteristic function of the cube $Q=[0,1]^{2d}$, when $a=(a_{kn})_{k,n\in\Z^d}\in\ell^{p,q}_{m_\lambda}$.
Equivalently, $F\in W(L^{p,q}_{m_\lambda})$ if and only if
\beqs
\label{5}
|F|\leq\sum_{k,n\in\Z^d}b_{kn}T_{(k,n)}\chi_Q
\eeqs
for some $b=(b_{kn})_{k,n\in\Z^d}\in\ell^{p,q}_{m_\lambda}$ (cf. \cite[pg. 222]{G}).
The amalgam space $W(L^{p,q}_{m_\lambda})$ is endowed with the norm
\beqsn
\|F\|_{W(L^{p,q}_{m_\lambda})}=\|a\|_{\ell^{p,q}_{m_\lambda}}.
\eeqsn

In what follows we shall need the Young estimate for $L^{p,q}_{m_\lambda}$:
\begin{Prop}
\label{prop1113G}
Let $\omega$ be a weight function and  $L$ as in \eqref{L}. Set, for every $\lambda\in\R$,
\beqsn
\mu(\lambda):=\begin{cases}
\lambda L,&\lambda\geq0\cr
\lambda/L,&\lambda<0,
\end{cases}
\qquad
\nu(\lambda):=\begin{cases}
\lambda L,&\lambda\geq0\cr
|\lambda|,&\lambda<0.
\end{cases}
\eeqsn
Then, for $F\in L^{p,q}_{m_{\mu(\lambda)}}$ and $G\in L^1_{m_{\nu(\lambda)}}$,
with $1\leq p,q\leq+\infty$, we have that $F*G\in L^{p,q}_{m_\lambda}$ and
\beqsn
\|F*G\|_{L^{p,q}_{m_\lambda}}\leq C_\lambda\|F\|_{L^{p,q}_{m_{\mu(\lambda)}}}
\|G\|_{L^1_{m_{\nu(\lambda)}}}
\eeqsn
for a constant $C_\lambda>0$ depending on $\lambda$.
\end{Prop}

\begin{proof}
Let us first assume $1\leq p,q<+\infty$.
From the definition of convolution
\beqsn
\|F*G\|_{L^{p,q}_{m_\lambda}}\le \left(\int_{\R^d}\left(\int_{\R^d}\left(e^{\lambda\omega(x,\xi)}\int_{\R^{2d}}
|F(x-y,\xi-\eta)G(y,\eta)|dyd\eta \right)^pdx\right)^{q/p}d\xi
\right)^{1/q}.
\eeqsn

Now, for $\lambda\geq0$ we have, by \eqref{L},
\beqsn
\lambda\omega(x,\xi)
\leq\lambda L(\omega(x-y,\xi-\eta)+\omega(y,\eta)+1),
\eeqsn
so that
\beqsn
\|F*G\|_{L^{p,q}_{m_\lambda}}
\le e^{\lambda L}
\left\|\left(|F|e^{\lambda L\omega(\cdot)}\right)*\left(|G|e^{\lambda L\omega(\cdot)}\right)
\right\|_{L^{p,q}}.
\eeqsn
By the standard Young's inequality for (non weighted) $L^{p,q}$ spaces we obtain
\beqsn
\|F*G\|_{L^{p,q}_{m_\lambda}}\leq&&
C_\lambda \||F|e^{\lambda L\omega(\cdot)}\|_{L^{p,q}}
\||G|e^{\lambda L\omega(\cdot)}\|_{L^1}\\
=&&C_\lambda\|F\|_{L^{p,q}_{m_{\lambda L}}}
\|G\|_{L^1_{m_{\lambda L}}}=C_\lambda\|F\|_{L^{p,q}_{m_{\mu(\lambda)}}}
\|G\|_{L^1_{m_{\nu(\lambda)}}}.
\eeqsn

For $\lambda<0$ we have, by \eqref{L},
\beqsn
\lambda\omega(x,\xi)\leq\frac\lambda L\omega(x-y,\xi-\eta)-\lambda\omega(y,\eta)-
\lambda
=\frac\lambda L\omega(x-y,\xi-\eta)+|\lambda|\omega(y,\eta)-
\lambda,
\eeqsn
and then, as before,
\beqsn
\|F*G\|_{L^{p,q}_{m_\lambda}}\leq C_\lambda\|F\|_{L^{p,q}_{m_{\lambda/L}}}
\|G\|_{L^1_{m_{|\lambda|}}}=
C_\lambda\|F\|_{L^{p,q}_{m_{\mu(\lambda)}}}
\|F\|_{L^1_{m_{\nu(\lambda)}}}
\eeqsn
for some $C_\lambda>0$.
The proof for $p=+\infty$ and/or $q=+\infty$ is similar.
\end{proof}

We have the following proposition, analogous to \cite[Prop. 11.1.4]{G}. We give the proof for the convenience of the reader.
\begin{Prop}
\label{prop1114G}
Let $\omega$ be a weight function, $L$ as in \eqref{L} and $\lambda>0$.
If $F\in W(L^{p,q}_{m_{\lambda L}})$ is continuous, then for every $\alpha,\beta>0$ there exists a constant $C_{\alpha,\beta,\lambda}>0$ such that
\beqsn
\left\|F|_{\alpha\Z^d\times\beta\Z^d}\right\|_{\ell^{p,q}_{\tilde m_\lambda}}
\leq C_{\alpha,\beta,\lambda}\|F\|_{W(L^{p,q}_{m_{\lambda L}})},
\eeqsn
for $\tilde m_\lambda(k,n):=m_\lambda(\alpha k,\beta n)$.
\end{Prop}

\begin{proof}
The continuity of $F$ is necessary in order that $F(\alpha k,\beta n)$ is well defined.
For $(\alpha k,\beta n)\in(r,s)+[0,1]^{2d}$ with $(r,s)\in\Z^d\times\Z^d$ we have 
\beqsn
\tilde m_\lambda(k,n)=&&e^{\lambda\omega(\alpha k,\beta n)}
\leq \sup_{(x,\xi)\in[0,1]^{2d}}e^{\lambda L(\omega(r,s)+\omega(x,\xi)+1)}\\
=&&e^{\lambda L}e^{\lambda L\omega(r,s)}\sup_{(x,\xi)\in[0,1]^{2d}}e^{\lambda L\omega(x,\xi)}
=C_\lambda m_{\lambda L}(r,s)
\eeqsn
for $C_\lambda=e^{\lambda L}\sup_{(x,\xi)\in[0,1]^{2d}}e^{\lambda L\omega(x,\xi)}$.
Then
\beqsn
|F(\alpha k,\beta n)|m_\lambda(\alpha k,\beta n)\leq&&
\esssup_{(x,\xi)\in[0,1]^{2d}}|F(r+x,s+\xi)|\cdot
C_\lambda m_{\lambda L}(r,s)\\
\leq&& C_\lambda\|F\cdot T_{(r,s)}\chi_Q\|_{L^\infty}\cdot m_{\lambda L}(r,s).
\eeqsn

Since there are at most $\tilde{C}_\alpha:=\left(\left[\frac1\alpha\right]+1\right)^d$ points
$\alpha k\in r+[0,1]^d$ we obtain
\beqsn
\left(\sum_{k\in\Z^d}|F(\alpha k,\beta n)|^pm_\lambda(\alpha k,\beta n)^p\right)^{1/p}
\!\!\!\!\leq\left(\tilde{C}_\alpha C_\lambda^p\sum_{r\in\Z^d}
\left\|F\cdot T_{(r,s)}\chi_Q\right\|_{L^\infty}^pm_{\lambda L}(r,s)^p\right)^{1/p}\!\!\!\!.
\eeqsn

Analogously, there are at most $\tilde{C}_\beta:=\left(\left[\frac 1\beta\right]+1\right)^d$ points
$\beta n\in s+[0,1]^d$ and therefore
\beqsn
\left\|F|_{\alpha\Z^d\times\beta\Z^d}\right\|_{\ell^{p,q}_{\tilde m_\lambda}}
\leq&&\left(\sum_{s\in\Z^d}\tilde{C}_\beta
\left(\tilde{C}_\alpha C_\lambda^p\sum_{r\in\Z^d}
\left\|F\cdot T_{(r,s)}\chi_Q\right\|_{L^\infty}^pm_{\lambda L}(r,s)^p\right)^{q/p}\right)^{1/q}\\
\leq&& \tilde{C}_\beta^{1/q}\tilde{C}_\alpha^{1/p}C_\lambda\| F\|_{W(L^{p,q}_{m_{\lambda L}})}.
\eeqsn
\end{proof}

\begin{Prop}
\label{propAle}
Let $\omega$ be a weight function, $L$ as in \eqref{L} and $\lambda>0$.
If $F\in L^{p,q}_{m_{\lambda L}}$ and $G\in L^1_{m_{\lambda L^2}}$, then
$F*G\in W(L^{\infty}_{m_\lambda})$ and
\beqsn
\|F*G\|_{W(L^{\infty}_{m_\lambda})}\leq C_\lambda
\|F\|_{L^\infty_{m_{\lambda L}}}\|G\|_{L^1_{m_{\lambda L^2}}}.
\eeqsn
\end{Prop}

\begin{proof}
From the definition of the norm in $W(L^{\infty}_{m_\lambda})$ we have
\beqsn
&&\|F*G\|_{W(L^\infty_{m_\lambda})}\\
=&&\sup_{k,n\in\Z^d}\left\{\left[\esssup_{(x,\xi)\in[0,1]^{2d}}\left|
\int_{\R^{2d}}F(x+k-y,\xi+n-\eta)G(y,\eta)dyd\eta\right|\right]e^{\lambda\omega(k,n)}
\right\}.
\eeqsn

By \eqref{L}, it is easy to see that 
\beqsn
\omega(k,n)\leq
L\omega(x+k-y,\xi+n-\eta)+L^2\omega(x,\xi)+L^2\omega(y,\eta)+L^2+L.
\eeqsn

Therefore, we obtain
\beqsn
\|F*G\|_{W(L^{\infty}_{m_\lambda})}
\leq e^{\lambda(L^2+L)}
\sup_{k,n\in\Z^d}\bigg\{\esssup_{(x,\xi)\in[0,1]^{2d}}\bigg|\int_{\R^{2d}}&&
e^{\lambda L\omega(x+k-y,\xi+n-\eta)}|F(x+k-y,\xi+n-\eta)|\\
&&\cdot\,
e^{\lambda L^2\omega(y,\eta)}|G(y,\eta)|dyd\eta\bigg|e^{\lambda L^2\omega(x,\xi)}
\bigg\}.
\eeqsn

Since $(x,\xi)\in[0,1]^{2d}$ we have that $e^{\lambda L^2\omega(x,\xi)}$ is bounded by
a constant depending on $\lambda$ (and $L$), so we obtain 
\beqsn
\|F*G\|_{W(L^{\infty}_{m_\lambda})}\leq&&
C_\lambda
\sup_{k,n\in\Z^d}\bigg\{\esssup_{(x,\xi)\in[0,1]^{2d}}
\bigg|\bigg(e^{\lambda L\omega(\cdot,\cdot)}|F(\cdot,\cdot)|\bigg)*
\bigg(e^{\lambda L^2\omega(\cdot,\cdot)}|G(\cdot,\cdot)|\bigg)(x+k,\xi+n)\bigg|\bigg\}\\
=&&C_\lambda\left\|\left(e^{\lambda L\omega} |F|\right)*\left(e^{\lambda L^2\omega} |G|\right)
\right\|_{L^\infty(\R^{2d})},
\eeqsn
for some $C_\lambda>0$.

By Young's inequality we finally deduce
\beqsn
\|F*G\|_{W(L^{\infty}_{m_\lambda})}\leq
C_\lambda\|e^{\lambda L\omega} F\|_{L^\infty}\|e^{\lambda L^2\omega}
G\|_{L^1}
=C_\lambda\|F\|_{L^{\infty}_{m_{\lambda L}}}\|G\|_{L^{1}_{m_{\lambda L^2}}}.
\eeqsn
\end{proof}

Now, our aim is  to show that there is an isomorphism between $\Sch_\omega(\R^d)$ and its
image through the analysis operator $C_{\varphi_0}$:
\beqs
\label{iso}
C_{\varphi_0}:\ \Sch_\omega(\R^d)\longrightarrow
 \Im C_{\varphi_0}\subseteq\tilde\Lambda_\omega,
\eeqs
where $\tilde\Lambda_\omega$ is defined  in \eqref{Lambdaomega}.

The following proposition holds for every window function $\varphi\in\Sch_\omega(\R^d)\setminus\{0\}$ and in particular for our fixed window $\varphi_0\in\Sch_\omega(\R^d)$:
\begin{Prop}
\label{prop1Ale}
Let $\omega$ be a weight function and $\varphi\in\Sch_{\omega}(\R^{d})\setminus \{0\}$.
The analysis operator
\beqsn
C_\varphi:\ \Sch_\omega(\R^d)\longrightarrow\tilde\Lambda_\omega
\eeqsn
is continuous.
\end{Prop}

\begin{proof}
It is known that if
$f\in\Sch_\omega(\R^d)$
then for every $\lambda>0$ there exists $C_\lambda>0$ such that
\beqsn
|V_\varphi f(z)|\leq C_\lambda e^{-\lambda\omega(z)},
\qquad z\in\R^{2d}.
\eeqsn
In fact, this property is proved in \cite{GZ} when $\omega$ is subadditive, but it is still true in the
general case (Theorem~\ref{propSomega}).
Since $C_\varphi f=(V_\varphi f(\sigma))_{\sigma\in\Lambda}$ we  have
$C_\varphi f\in\tilde\Lambda_\omega$.

Now, we prove that the operator $C_{\varphi}$ is continuous.  By \cite[Lemma 11.3.3]{G}
\beqsn
|V_\varphi f(z)|\leq\frac{1}{(2\pi)^d\|\varphi\|_{L^2}^2}(|V_\varphi f|*|V_\varphi \varphi|)(z),
\qquad\forall z\in\R^{2d}.
\eeqsn
By Propositions~\ref{prop1114G} and \ref{propAle}, for every fixed $\lambda>0$ we
obtain
\beqsn
\sup_{\sigma\in\Lambda}|V_\varphi f(\sigma)|e^{\lambda\omega(\sigma)}
=&&\|V_\varphi f|_{\alpha_0\Z^d\times\beta_0\Z^d}\|_{\ell^\infty_{\tilde{m}_\lambda}}
\leq C_\lambda\|V_\varphi f\|_{W(L^{\infty}_{m_{\lambda L}})}\\
\leq&&C'_\lambda\|V_\varphi f\|_{L^\infty_{m_{\lambda L^2}}}\|V_\varphi\varphi\|_{L^1_{m_{\lambda L^3}}}
\eeqsn
for $\tilde{m}_\lambda (k,n)=m_\lambda(\alpha_0k,\beta_0n)$ and for
some $C_\lambda,C'_\lambda>0$ ($\alpha_0$ and $\beta_0$ are fixed).
Observe that, since $f,\varphi\in\Sch_\omega(\R^d)$, then $V_\varphi f\in
L^\infty_{m_{\lambda L^2}}$ and $V_\varphi\varphi\in L^1_{m_{\lambda L^3}}$ for every
$\lambda>0$ by Theorem~\ref{propSomega}(h).

Therefore, for every fixed $\lambda>0$ there exists a constant
$C''_\lambda=C'_\lambda\|V_\varphi\varphi\|_{L^1_{m_{\lambda L^3}}}>0$
such that
\beqsn
\sup_{\sigma\in\Lambda}|V_\varphi f(\sigma)|e^{\lambda\omega(\sigma)}
\leq C''_\lambda\|V_\varphi f\|_{L^\infty_{m_{\lambda L^2}}}.
\eeqsn
This gives the continuity by Theorem~\ref{propSomega}(h).
\end{proof}

The following proposition is valid for any $\psi\in\Sch_\omega(\R^d)\setminus\{0\}$.
\begin{Prop}
\label{prop2Ale}
Let $\omega$ be a weight function and $\psi\in\Sch_\omega(\R^d)\setminus\{0\}$.
Then the synthesis operator
\beqsn
D_\psi:\ \tilde\Lambda_\omega\longrightarrow\Sch_\omega(\R^d)
\eeqsn
 is continuous.
\end{Prop}

\begin{proof}
Let $c=(c_{\sigma})_{\sigma\in\Lambda}\in\tilde\Lambda_\omega$. For simplicity, we denote $c_{\sigma}$ by $c_{kn}$ for $\sigma=(\alpha_{0} k,\beta_{0}n)$. We start proving that
$D_\psi c\in\Sch_\omega(\R^d)$.
We shall apply Theorem~\ref{propSomega}(c) with $p=+\infty$.
So, first, we have to see that $D_\psi c\in\Sch(\R^d)$.

By definition
\beqs
\label{A0}
(D_\psi c)(t)=\sum_{k,n\in\Z^d}c_{kn}M_{\beta_0n}T_{\alpha_0 k}\psi(t)
=\sum_{k,n\in\Z^d}c_{kn}e^{i\langle\beta_0n,t\rangle}\psi(t-\alpha_0 k).
\eeqs
Now, we see that $D_\psi c\in C^\infty(\R^d)$. To that aim we show that for each $\gamma\in\N_{0}^{d}$, the series
\beqs
\label{A1}
\sum_{k,n\in\Z^d}\partial_t^\gamma\left[c_{kn}e^{i\langle\beta_0n,t\rangle}\psi(t-\alpha_0k)\right]
\eeqs
is uniformly convergent on $t\in\R^d$.
Let us compute
\beqs
\nonumber
\partial_t^\gamma\left[c_{kn}e^{i\langle\beta_0n,t\rangle}\psi(t-\alpha_0k)\right]&=&\sum_{\mu\leq\gamma}\binom\gamma\mu c_{kn}\partial_t^\mu\left(e^{i\langle\beta_0n,t\rangle}
\right)
\partial_t^{\gamma-\mu}\psi(t-\alpha_0k)\\
\label{A2}
&=&c_{kn}\sum_{\mu\leq\gamma}\binom\gamma\mu(i\beta_0n)^\mu
e^{i\langle\beta_0n,t\rangle}
\partial_t^{\gamma-\mu}\psi(t-\alpha_0k).
\eeqs

Since $(c_{kn})_{k,n\in\Z^d}\in\tilde\Lambda_\omega$, for every $\lambda>0$ there
exists $C_\lambda>0$ such that
\beqsn
|c_{kn}|\leq C_\lambda e^{-\lambda\omega(\alpha_0k,\beta_0n)},\qquad k,n\in\Z^d.
\eeqsn
Now, since $\omega$ is increasing it is obvious that
$\omega(t,s)\geq\frac12(\omega(t)+\omega(s)).$
%
Therefore
\beqs
\label{A2''}
|c_{kn}|\leq&&C_\lambda e^{-\lambda\omega(\alpha_0k,\beta_0n)}
\leq C_\lambda e^{-\frac{\lambda}{2 }\omega(\alpha_0 k)}e^{-\frac{\lambda}{2 }
\omega(\beta_0n)}.
\eeqs
Since
\beqsn
\omega(t)\leq L(\omega(\alpha_0k-t)+\omega(\alpha_0k)+1),
\eeqsn
we obtain
\beqs
\nonumber
|c_{kn}|\leq&&C_\lambda e^{-\frac{\lambda}{2 }\omega(\beta_0n) }e^{-\frac{\lambda}{4}
\omega(\alpha_0k)}e^{-\frac{\lambda}{4}
\omega(\alpha_0k)}\\
\label{A2'}
\leq&&C_\lambda e^{-\frac{\lambda}{2 }\omega(\beta_0n)} e^{-\frac{\lambda}{4}
\omega(\alpha_0k)}
e^{-\frac{\lambda}{4}\left[\frac1L\omega(t)-\omega(\alpha_0k-t)-1\right]}\\
\nonumber
\leq&&C_\lambda e^{-\frac{\lambda}{4L}\omega(0)}e^{\frac{\lambda}{4}}
e^{-\frac{\lambda}{2 }\omega(\beta_0n)}e^{-\frac{\lambda}{4}\omega(\alpha_0k)}
e^{\frac{\lambda}{4}\omega(\alpha_0k-t)}\\
\nonumber
=&&C'_\lambda e^{-\frac{\lambda}{2 }\omega(\beta_0n)}
e^{-\frac{\lambda}{4}\omega(\alpha_0k)}
e^{\frac{\lambda}{4}\omega(\alpha_0k-t)}.
\eeqs

Then we have, by \eqref{A2}, for $C_{\lambda,\gamma}=C'_\lambda\max_{\mu\leq\gamma}
|\beta_0|^{|\mu|-|\gamma|}$, since $\psi\in\Sch_\omega(\R^d)$ 
(see Definition~\ref{Somega}),
\beqs
\nonumber
\lefteqn{\left|\partial_t^\gamma\left[c_{kn}e^{i\langle\beta_0n,t\rangle}\psi(t-\alpha_0k)\right]\right|
\leq|c_{kn}|\sum_{\mu\leq\gamma}\binom{\gamma}{\mu}|\beta_0n|^{|\mu|}
|\partial_t^{\gamma-\mu}\psi(t-\alpha_0k)|}\\
\nonumber
&&\quad =|c_{kn}|\sum_{\mu\leq\gamma}\binom{\gamma}{\mu}|\beta_0|^{|\mu|-|\gamma|}
|\beta_0|^{|\gamma|}n^{|\mu|}
|\partial_t^{\gamma-\mu}\psi(t-\alpha_0k)|\\
\label{A3}
&&\quad \leq 
\sum_{\mu\leq\gamma}\binom{\gamma}{\mu}C_{\lambda,\gamma} e^{-\frac{\lambda}{2 }\omega(\beta_0n)}
e^{-\frac{\lambda}{4}\omega(\alpha_0k)}|\beta_0n|^{|\gamma|}
\left|\partial_t^{\gamma-\mu}\psi(t-\alpha_0k)\right|e^{\frac{\lambda}{4}\omega(t-\alpha_0k)}\\
\nonumber
&&\quad \leq C'_{\lambda,\gamma}e^{-\frac{\lambda}{4}\omega(\alpha_0k)}|\beta_0n|^{|\gamma|}
e^{-\frac{\lambda}{2 }\omega(\beta_0n)}
\eeqs
for some $C'_{\lambda,\gamma}>0$. Hence, for $\lambda>0$ sufficiently large the series
\beqsn
\sum_{k,n\in\Z^d} \partial_t^\gamma\left[c_{kn}e^{i\langle\beta_0n,t\rangle}
\psi(t-\alpha_0k)\right]
\eeqsn
is uniformly convergent on $t\in\R^d$.
This implies that $D_\psi c\in C^\infty(\R^d)$ for every $c\in\tilde\Lambda_\omega$.

In particular we can differentiate $D_\psi c$ in \eqref{A0} term by term, so that, to prove that
$D_\psi c\in\Sch(\R^d)$, we can estimate, for every $\gamma,\mu\in\N_0^d$,
\beqsn
\lefteqn{\left|t^\mu \partial_t^\gamma(D_\psi c)\right|=
\bigg|t^\mu
\sum_{k,n\in\Z^d} \partial_t^\gamma\left[c_{kn}e^{i\langle\beta_0n,t\rangle}
\psi(t-\alpha_0k)\right]\bigg|}\\
&& \quad \leq |t|^{|\mu|}
\sum_{k,n\in\Z^d}C_{\lambda,\gamma}\sum_{\tilde\mu\leq\gamma}
\binom{\gamma}{\tilde\mu}e^{-\frac{\lambda}{2 }\omega(\beta_0n)}
e^{-\frac{\lambda}{4}\omega(\alpha_0k)}
|\beta_0n|^{|\gamma|}\left|\partial_t^{\gamma-\tilde\mu}\psi(t-\alpha_0k)\right|
e^{\frac{\lambda}{4}\omega(t-\alpha_0k)}.
\eeqsn

Since
\beqsn
|t|^{|\mu|}\leq&&(|t-\alpha_0k|+|\alpha_0k|)^{|\mu|}
\leq  2^{|\mu|}(1+|t-\alpha_0k|^{|\mu|})(1+|\alpha_0k|^{|\mu|}),
\eeqsn
we  obtain 
\beqs
\nonumber
\left|t^\mu \partial_t^\gamma(D_\psi c)\right|\leq&&\sum_{k,n\in\Z^d}2^{|\mu|}C_{\lambda,\gamma}(1+|\alpha_0k|^{|\mu|})
e^{-\frac{\lambda}{4}\omega(\alpha_0k)}|\beta_0n|^{|\gamma|}
e^{-\frac{\lambda}{2 }\omega(\beta_0n)}\\
\nonumber
&&\cdot\sum_{\tilde\mu\leq\gamma}\binom{\gamma}{\tilde\mu}
(1+|t-\alpha_0k|^{|\mu|})|\partial_t^{\gamma-\tilde\mu}\psi(t-\alpha_0k)|
e^{\frac{\lambda}{4}\omega(t-\alpha_0k)}\\
\label{A41}
\leq&& C_{\lambda,\gamma,\mu}\sum_{k,n\in\Z^d}(1+|\alpha_0k|^{|\mu|})
e^{-\frac{\lambda}{4}\omega(\alpha_0k)}|\beta_0n|^{|\gamma|}
e^{-\frac{\lambda}{2 }\omega(\beta_0n)},
\eeqs
for some $C_{\lambda,\gamma,\mu}>0$  because $\psi\in\Sch_\omega(\R^d)$,
by Theorem~\ref{propSomega}(b).
Since the series in \eqref{A41} converges for  $\lambda>0$ sufficiently large, we have
 $D_\psi c\in\Sch(\R^d)$.

By Theorem~\ref{propSomega}(c), to see that $D_\psi c\in\Sch_\omega(\R^d)$
it is now enough to prove that, for every $\tilde\lambda>0$, the following two conditions hold:
\beqs
\label{A5}
&&\sup_{t\in\R^d}e^{\tilde\lambda\omega(t)}|D_\psi c(t)|<+\infty,\\
\label{A6}
&&\sup_{\xi\in\R^d}e^{\tilde\lambda\omega(\xi)}|\widehat{D_\psi c}(\xi)|<+\infty.
\eeqs

To prove \eqref{A5} we use the calculations in \eqref{A2'} and obtain, for every
$\lambda\geq4L\tilde\lambda$,
\beqs
\nonumber
e^{\tilde\lambda\omega(t)}|D_\psi c(t)|\leq&&
\nonumber
e^{\tilde\lambda\omega(t)}\sum_{k,n\in\Z^d}|c_{kn}||\psi(t-\alpha_0k)|\\
\label{A7}
\leq&&\sum_{k,n\in\Z^d}C_\lambda e^{-\frac{\lambda}{2 }\omega(\beta_0n)}
e^{-\frac{\lambda}{4}\omega(\alpha_0k)}e^{\tilde\lambda\omega(t)}
e^{-\frac{\lambda}{4L}\omega(t)}e^{\frac{\lambda}{4}}
e^{\frac{\lambda}{4}\omega(\alpha_0k-t)}
|\psi(t-\alpha_0k)|\\
\label{A71}
\leq&&\tilde{C}_\lambda e^{-\left(\frac{\lambda}{4L}-\tilde\lambda\right)\omega(0)}
\sum_{k,n\in\Z^d}e^{-\frac{\lambda}{2}\omega(\beta_0n)}e^{-\frac{\lambda}{4}\omega(\alpha_0k)},
\eeqs
for some $\tilde C_\lambda>0$,
since $\psi\in\Sch_\omega(\R^d)$.
For $\lambda$ sufficiently large the series in \eqref{A71} converges and hence \eqref{A5}
is proved.

To prove \eqref{A6} let us now consider
\beqsn
\widehat{D_\psi c}(\xi)=
\int_{\R^d}e^{-i\langle t,\xi\rangle}\sum_{k,n\in\Z^d}c_{kn}e^{i\langle\beta_0n,t\rangle}
\psi(t-\alpha_0k)dt.
\eeqsn

Since the series
\beqsn
e^{-i\langle t,\xi\rangle}\sum_{k,n\in\Z^d}c_{kn}
e^{i\langle\beta_0n,t\rangle}\psi(t-\alpha_0k)
\eeqsn
converges uniformly and moreover, by \eqref{A7} with $\tilde\lambda=0$ and $\lambda$ large enough,
\beqs
\nonumber
\big|e^{-i\langle t,\xi\rangle}&&\!\!\!\!\!\sum_{k,n\in[-N,N]^d}c_{kn}
e^{i\langle\beta_0n,t\rangle}\psi(t-\alpha_0k)\big|
\leq\sum_{k,n\in[-N,N]^d}|c_{kn}||\psi(t-\alpha_0k)|\\
\label{A72}
&\leq& \sum_{k,n\in\Z^d}\tilde{C}_\lambda
e^{\frac{\lambda}{4}}e^{-\frac{\lambda}{2 }\omega(\beta_0n)}
e^{-\frac{\lambda}{4}\omega(\alpha_0k)}e^{-\frac{\lambda}{4L}\omega(t)}\\
\nonumber
&\leq& \tilde{C}'_\lambda e^{-\frac{\lambda}{4L}\omega(t)}\in L^1(\R^d),
\eeqs
 by the Dominated Convergence Theorem
\beqsn
\widehat{D_\psi c}(\xi)=&&\sum_{k,n\in\Z^d}c_{kn}\int_{\R^d}
e^{-i\langle t,\xi\rangle}e^{i\langle\beta_0n,t\rangle}\psi(t-\alpha_0k)dt\\
=&&\sum_{k,n\in\Z^d}c_{kn}\int_{\R^d}
e^{-i\langle t+\alpha_0k,\xi-\beta_0n\rangle}\psi(t)dt\\
=&&\sum_{k,n\in\Z^d}c_{kn}e^{-i\langle\alpha_0k,\xi-\beta_0n\rangle}
\hat{\psi}(\xi-\beta_0n).
\eeqsn

Then
\beqs
\label{A8}
\left|e^{\tilde\lambda\omega(\xi)}\widehat{D_\psi c}(\xi)\right|
\leq e^{\tilde\lambda\omega(\xi)}
\sum_{k,n\in\Z^d}|c_{kn}|
|\hat{\psi}(\xi-\beta_0n)|
\eeqs
and since $\hat\psi\in\Sch_\omega(\R^d)$ satisfies the same estimates as $\psi$
the proof of \eqref{A6} is similar to that of \eqref{A5} and so
$D_\psi c\in\Sch_\omega(\R^d)$.

Now, we see that $D_\psi$ is continuous. To this aim we have to
estimate \eqref{A5} and \eqref{A6}, for every $\tilde\lambda>0$, by some seminorm of $c=(c_{kn})_{k,n\in\Z^d}$ in $\tilde\Lambda_\omega$.
Writing, for every $\lambda>0$,
\beqsn
|c_{kn}|\leq\sup_{k,n\in\Z^d}\left(|c_{kn}|e^{\lambda\omega(\alpha_0k,\beta_0n)}\right)
\cdot e^{-\lambda\omega(\alpha_0k,\beta_0n)},
\eeqsn
and proceeding as to obtain \eqref{A71},
with $\sup_{k,n\in\Z^d}\left(|c_{kn}|e^{\lambda\omega(\alpha_0k,\beta_0n)}\right)$ instead of
$C_\lambda$ in \eqref{A2''},
 we obtain that for every $\tilde\lambda>0$ there
exist $\lambda>0$ and $C_{\tilde\lambda}>0$ such that
\beqsn
\sup_{t\in\R^d}e^{\tilde\lambda\omega(t)}|D_\psi c(t)|
\leq C_{\tilde\lambda}\sup_{k,n\in\Z^d}(|c_{kn}|e^{\lambda\omega(\alpha_0k,\beta_0n)}).
\eeqsn
Similarly, from \eqref{A8},
\beqsn
\sup_{\xi\in\R^d}e^{\tilde\lambda\omega(\xi)}|\widehat{D_\psi c}(\xi)|
\leq C'_{\tilde\lambda}\sup_{k,n\in\Z^d}(|c_{kn}|e^{\lambda\omega(\alpha_0k,\beta_0n)}),
\eeqsn
for some $C'_{\tilde\lambda}>0$.
Therefore $D_\psi$ is continuous and the proof is complete.
\end{proof}

We already know from the general theory of Gabor frames that
$D_{\psi_0}C_{\varphi_0}=\Id$ on
$\Sch_\omega(\R^d)$, as already observed in \eqref{add1}. Hence the operator in
\eqref{iso} is injective, surjective, continuous and its inverse
$D_{\psi_0}|_{\Im C_{\varphi_0}}$  is continuous. Since we consider
on $\Im C_{\varphi_0}$ the topology induced by
$\tilde\Lambda_\omega$, to see that $\Sch_\omega(\R^d)$ is nuclear it is enough to
check that $\tilde\Lambda_\omega$
is nuclear \cite[Prop. 28.6]{MV}.

\section{Nuclearity of $\Sch_\omega(\R^d)$}
\label{sec3}

In this section we show that  $\tilde\Lambda_\omega$ is nuclear by an application of Grothendieck-Pietsch criterion.
For a countable lattice $\Lambda$, we consider
a matrix
\beqs
\label{matrixA}
A=(a_{\sigma, k})_{\afrac{\sigma\in\Lambda,}{k\in\N}}
\eeqs
of K\"othe type with positive entries, in the sense that $A$ satisfies
\beqs
\label{k1}
&&a_{\sigma, k}>0\qquad\qquad\!\forall\sigma\in\Lambda, k\in\N,\\
\label{k2}
&&a_{\sigma,k}\leq a_{\sigma,k+1}\qquad\forall\sigma\in\Lambda, k\in\N.
\eeqs

We denote
\beqsn
&&\tilde\lambda^p(A):=
\Big\{c=(c_\sigma)_{\sigma\in\Lambda}:\
\|c\|_k:=\bigg(\sum_{\sigma\in\Lambda}|c_\sigma|^pa^p_{\sigma,k}\bigg)^{1/p}
<+\infty,\ \forall k\in\N\Big\},\quad 1\leq p<+\infty,\\
&&\tilde\lambda^\infty(A):=
\Big\{c=(c_\sigma)_{\sigma\in\Lambda}:\
\|c\|_k:=\sup_{\sigma\in\Lambda}|c_\sigma|a_{\sigma,k}<+\infty,\ \forall k\in\N\Big\}\\
&&\tilde c_0(A):=\Big\{c\in\tilde\lambda^\infty(A):\ \lim_{|\sigma|\to+\infty}|c_\sigma|
a_{\sigma,k}=0,\ \forall k\in\N\Big\}.
\eeqsn
 
 We put
 \beqsn
 \tilde\ell^p:=&&
\Big\{c=(c_\sigma)_{\sigma\in\Lambda}:\
\bigg(\sum_{\sigma\in\Lambda}|c_\sigma|^p\bigg)^{1/p}
<+\infty,\ \forall k\in\N\Big\}, \qquad 1\leq p<+\infty.
\eeqsn
Analogously, we define $\tilde\ell^\infty$ and $\tilde c_0$.
The spaces $\tilde\ell^p$, for $1\leq p\leq+\infty$, and $\tilde c_0$
are Banach spaces, while $\tilde\lambda^p(A)$, for $1\leq p\leq+\infty$, and  $\tilde c_0(A)$ are Fr\'echet spaces.
%
%
We consider the
canonical basis $(e_\eta)_{\eta\in\Lambda}$:
\beqsn
e_\eta=(\delta_{\eta\sigma})_{\sigma\in\Lambda}=\begin{cases}
1,&\sigma=\eta\cr
0,&\sigma\neq\eta.
\end{cases}
\eeqsn

%
%

Since $\Lambda$ is countable, it is obvious that $(e_\eta)_{\eta\in\Lambda}$ is a Schauder basis for $\tilde c_0(A)$ and $\tilde\lambda^p(A)$, for $1\leq p<+\infty$.

The following result is analogous to \cite[Prop. 28.16]{MV}. We give the proof in the case of lattices for the sake of completeness.
\begin{Th}
\label{prop2816MV}
Let $A$ be as in \eqref{matrixA} a matrix of K\"othe type with positive entries.
The following are equivalent:
\begin{itemize}
\item[(a)]
$\tilde\lambda^p(A)$ is nuclear for some $1\leq p\leq+\infty$;
\item[(b)]
$\tilde\lambda^p(A)$ is nuclear for all $1\leq p\leq+\infty$;
\item[(c)]
$\forall k\in\N\,\exists m\in\N,m\geq k\ \mbox{s.t.}\
\sum_{\sigma\in\Lambda}a_{\sigma,k}a_{\sigma,m}^{-1}<+\infty$.
\end{itemize}
\end{Th}

\begin{proof}
If $1\leq p<+\infty$, then $\tilde\lambda^p(A)$ is a Fr\'echet space 
with the increasing fundamental system of seminorms $(\|\cdot\|_m)_{m\in\N}$ and
the Schauder basis $(e_\eta)_{\eta\in\Lambda}$. We can then apply
Grothendieck-Pietsch criterion (see \cite[Thm. 28.15]{MV} or \cite{Pi}) to 
$\tilde\lambda^p(A)$
and obtain that $\tilde\lambda^p(A)$ is nuclear if and only if
\beqs
\label{condGP}
\forall k\in\N\ \exists m\in\N,m\geq k:\ 
\sum_{\sigma\in\Lambda}\|e_\sigma\|_k\|e_\sigma\|_m^{-1}<+\infty.
\eeqs
Since 
\beqsn
\|e_\sigma\|_k=\bigg(\sum_{\eta\in\Lambda}|\delta_{\sigma\eta}|^pa_{\eta,k}^p\bigg)^{1/p}
=a_{\sigma,k},
\eeqsn
the thesis is clear for $p<+\infty$.

Now, we treat the case $p=+\infty$. Assume that
$\tilde\lambda^\infty(A)$ is nuclear. We prove that
\beqs
\label{star}
\forall k\in\N\ \exists m\in\N,m\geq k:\
\lim_{|\sigma|\to+\infty}a_{\sigma,k}a_{\sigma,m}^{-1}=0.
\eeqs
To this aim, for every $k\in\N$, we denote
\beqsn
E_k:=\Big\{c=(c_\sigma)_{\sigma\in\Lambda}:\
\|c\|_k=\sup_{\sigma\in\Lambda}|c_\sigma|a_{\sigma,k}<+\infty\Big\}
\eeqsn
 the local space 
of $\tilde\lambda^\infty(A)$. This is a Banach space with the norm $\|\cdot\|_k$ (observe that $a_{\sigma,k}>0$ for all
$\sigma\in\Lambda,k\in\N$). The operator
\beqsn
A_k:\quad  E_k&&\longrightarrow\tilde\ell^\infty\\
c=(c_\sigma)_{\sigma\in\Lambda}&&\longmapsto
A_k(c):=(c_\sigma a_{\sigma,k})_{\sigma\in\Lambda}
\eeqsn
 is an isometric isomorphism and 
$A_k(E_k)=\tilde\ell^\infty$. For every $k\in\N$, the inclusion
\beqsn
i_k:\quad \tilde\lambda^\infty(A)&&\longrightarrow E_k\\
(c_\sigma)_{\sigma\in\Lambda}&&\longmapsto(c_\sigma)_{\sigma\in\Lambda}
\eeqsn
is compact by \cite[Lemma~24.17]{MV}. Indeed, $\tilde\lambda^\infty(A)$ is a locally convex space, which is nuclear (by assumption) and hence Schwartz by \cite[Cor. 28.5]{MV};
moreover
$E_k$ is a Banach space and hence we can apply \cite[Lemma~24.17]{MV}
and obtain that there exists a neighbourhood $V$ of 0 in $\tilde\lambda^\infty(A)$,
that we can take
of the form $\{c\in\tilde\lambda^\infty(A):\ \|c\|_m<\varepsilon\}$, for some $\varepsilon>0$
and with $m\geq k$ (the family of seminorms $(\|\cdot\|_m)_{m\in\N}$
is increasing), whose image through $i_k$ is precompact, and hence compact.
Moreover, for $m\geq k$ clearly  $E_m\subseteq E_k$. So,
 for every $k\in\N$ there exists $m\geq k$ such that the inclusion $i^k_m=i_k|_{E_m}$
\beqsn
i^k_m:\qquad E_m&&\longrightarrow E_k\\
(c_\sigma)_{\sigma\in\Lambda}&&\longmapsto(c_\sigma)_{\sigma\in\Lambda}
\eeqsn
is compact (and also $i^k_{m'}$ for all $m'\geq m$).

Then, we put $D:=A_k\circ i^k_m\circ A_m^{-1}$:
\beqsn
D:\qquad\tilde\ell^\infty&&\longrightarrow\tilde\ell^\infty\\
(c_\sigma)_{\sigma\in\Lambda}&&\longmapsto
(c_\sigma a_{\sigma,m}^{-1}a_{\sigma,k})_{\sigma\in\Lambda}.
\eeqsn
The operator $D$ is clearly compact.
The restriction $\tilde D:=D\big|_{\tilde c_0}$ satisfies $\tilde D(\tilde c_0)\subseteq \tilde c_0$, for $m\geq k$,
since
\beqsn
|c_\sigma|a_{\sigma,m}^{-1}a_{\sigma,k}
\leq|c_\sigma|a_{\sigma,m}^{-1}a_{\sigma,m}=|c_\sigma|\to0,
\eeqsn
for $c=(c_\sigma)_{\sigma\in\Lambda}\in\tilde c_0$.
The operator $\tilde D$ is also compact.

For every $\varepsilon>0$ we define, for $m\ge k,$
\beqsn
I_\varepsilon:=\{\sigma\in\Lambda:\ a_{\sigma,k}a_{\sigma,m}^{-1}\geq\varepsilon\},
\eeqsn
and also
\beqsn
T_\varepsilon:\quad \tilde c_0&&\longrightarrow\tilde c_0\\
c=(c_\sigma)_{\sigma\in\Lambda}&&
\longmapsto(T_\varepsilon(c))_{\sigma\in\Lambda}
=\begin{cases}
c_\sigma a_{\sigma,k}^{-1}a_{\sigma,m},&\sigma\in I_\varepsilon\cr
0,& otherwise.
\end{cases}
\eeqsn

The operator $T_\varepsilon:\,\tilde c_0\to\tilde c_0$ is continuous since
\beqsn
\sup_{\sigma\in\Lambda}|(T_\varepsilon(c))_\sigma|\leq
\frac1\varepsilon\sup_{\sigma\in\Lambda}|c_\sigma|.
\eeqsn

Now we consider
\beqsn
P_\varepsilon:=\tilde DT_\varepsilon:\
\tilde c_0&&\longrightarrow\tilde c_0\\
c=(c_\sigma)_{\sigma\in\Lambda}&&
\longmapsto\tilde c=(\tilde c_\sigma)_{\sigma\in\Lambda}
=\begin{cases}
c_\sigma, &\sigma\in I_\varepsilon\cr
0,&\sigma\in\Lambda\setminus I_\varepsilon.
\end{cases}
\eeqsn
Hence, $P_\varepsilon$ is a compact projection on
\beqsn
S_\varepsilon:=\{(c_\sigma)_{\sigma\in\Lambda}\in\tilde c_0:\
c_\sigma=0\ \mbox{for}\ \sigma\in\Lambda\setminus I_\varepsilon\}
\subseteq\tilde c_0.
\eeqsn
Since $\tilde c_0$ is a Banach space we can apply \cite[Cor. 15.6]{MV} and
obtain that the kernel $\ker(\Id-P_\varepsilon)$ is finite dimensional.
But $P_\varepsilon$ is a
projection and hence its image 
$
\Im(P_\varepsilon)=\ker(\Id-P_\varepsilon)
$
 is finite dimensional and
 $I_\varepsilon$ must be finite for every $\varepsilon>0$.
Then
\beqsn
\lim_{|\sigma|\to+\infty}a_{\sigma,k}a_{\sigma,m}^{-1}=0
\eeqsn
and \eqref{star} is proved.

This implies that $\tilde\lambda^\infty(A)=\tilde c_0(A)$.  Indeed, if
$c=(c_\sigma)_{\sigma\in\Lambda}\in\tilde\lambda^\infty(A)$ then for every $k\in\N$ we find
$m\in\N$, $m\geq k$ such that
\eqref{star} holds and we get 
\beqsn
\lim_{|\sigma|\to+\infty}|c_\sigma|a_{\sigma,k}=
\lim_{|\sigma|\to+\infty}|c_\sigma|a_{\sigma,m}a_{\sigma,k}a_{\sigma,m}^{-1}=0,
\eeqsn
since $|c_\sigma|a_{\sigma,m}$ is bounded because $c\in\tilde\lambda^\infty(A)$ and
$a_{\sigma,k}a_{\sigma,m}^{-1}\to0$ by \eqref{star}. Therefore $c\in\tilde c_0(A)$.

Now, $\tilde c_0(A)$ is a Fr\'echet space
endowed with the increasing fundamental system of seminorms $(\|~\cdot~\|_m)_{m\in\N}$ and
the Schauder basis $(e_\eta)_{\eta\in\Lambda}$. We can then apply
Grothendieck-Pietsch criterion  \eqref{condGP} to $\tilde c_0(A)$
 for
\beqsn
\|e_\sigma\|_k=\sup_{\eta\in\Lambda}|\delta_{\sigma \eta}|a_{\eta,k}=a_{\sigma,k}.
\eeqsn
Since $\tilde c_0(A)=\tilde\lambda^\infty(A)$ is nuclear by assumption, then
\eqref{condGP} implies $(c)$.

On the contrary, if $(c)$ holds then $\tilde c_0(A)$ is nuclear by the
Grothendiech-Pietsch criterion \eqref{condGP}.
We see again that $\tilde\lambda^\infty(A)=\tilde c_0(A)$.
If
$c=(c_\sigma)_{\sigma\in\Lambda}\in\tilde\lambda^\infty(A)$, we have
\beqsn
|c_\sigma|a_{\sigma,k}=|c_\sigma|a_{\sigma,m}a_{\sigma,k}a_{\sigma,m}^{-1}
\longrightarrow0,
\eeqsn
since $|c_\sigma|a_{\sigma,m}$ is bounded for $c\in\tilde\lambda^\infty(A)$
and \eqref{star} holds by the convergence of the series in $(c)$.
Therefore $c\in\tilde c_0(A)$.
\end{proof}

Observe that, for $\Lambda=\alpha_0\Z^d\times\beta_0\Z^d$ as fixed
in Section~\ref{sec2}, the matrix
\beqs
\label{matrixAtilde}
\tilde{A}=(e^{k\omega(\sigma)})_{\afrac{\sigma\in\Lambda,}{k\in\N}}
\eeqs
satisfies \eqref{k1} and \eqref{k2}. Hence the space $\tilde\Lambda_\omega$
defined in \eqref{Lambdaomega} is, in fact,
%
\beqsn
\tilde\lambda^\infty(\tilde A):=
\Big\{c=(c_\sigma)_{\sigma\in\Lambda}:\
\|c\|_k:=\sup_{\sigma\in\Lambda}|c_\sigma|e^{k\omega(\sigma)}<+\infty,\ \forall k\in\N\Big\}.
\eeqsn

\begin{Prop}
\label{cor1}
The sequence space $\tilde\Lambda_\omega$ is nuclear.
\end{Prop}

\begin{proof}
By Theorem~\ref{prop2816MV} we have that
$\tilde\Lambda_\omega=\tilde\lambda^\infty(\tilde{A})$
 is nuclear if and only if
\beqs
\label{starstar}
\forall k\in\N\ \exists m\in\N,m\geq k, \
\mbox{s.t.}\
\sum_{\sigma\in\Lambda}e^{k\omega(\sigma)-m\omega(\sigma)}<+\infty.
\eeqs
Since, by condition $(\gamma)$ of
Definition~\ref{defomega},
\beqsn
e^{k\omega(\sigma)-m\omega(\sigma)}
\leq e^{-(m-k)a}e^{-(m-k)b\log(1+|\sigma|)}
=e^{-(m-k)a}\frac{1}{(1+|\sigma|)^{b(m-k)}},
\eeqsn
we have, for $m>k+\frac{2d}{b}$,
\beqsn
\sum_{\sigma\in\Lambda}\frac{1}{(1+|\sigma|)^{b(m-k)}}<+\infty.
\eeqsn
\end{proof}

As we explained at the end of Section~\ref{sec2}, we deduce:

\begin{Th}
\label{cor3}
The space $\Sch_\omega(\R^d)$ is nuclear.
\end{Th}


\section{Nuclearity of $\Sch_{(M_p)}(\R^d)$ with $L^2$ norms}
\label{sec4}

Let $(M_p)_{p\in\N_0}$ be a sequence
such that 
$M_p^{1/p}\to+\infty$ as $p\to+\infty$
 and consider
the locally convex space of rapidly decreasing ultradifferentiable functions
\beqs
\label{SMp}
\qquad
\Sch_{(M_p)}(\R^d):=\Big\{f\in C^\infty(\R^d): \
\sup_{\alpha,\beta\in\N_0^d}\sup_{x\in\R^d}\frac{j^{|\alpha+\beta|}}{M_{|\alpha+\beta|}}
\|x^\alpha \partial^\beta f(x)\|_2<+\infty,\ \forall\, j\in\N \Big\},
\eeqs
where $\|\cdot\|_2$ denotes the $L^2$~norm. We write the associated function in the usual way:
\beqs
\label{Mt}
M(t)=\sup_{p\in\N}\log\frac{t^pM_0}{M_p}.
\eeqs

 Langenbruch~\cite{L} uses \eqref{12L} to show that the Hermite functions
$H_\gamma$, for $\gamma\in\N_0^d$, are an absolute Schauder basis in $\Sch_{(M_p)}(\R^d)$, where
\beqsn
H_\gamma(x):=(2^{|\gamma|}\gamma!\pi^{d/2})^{-1/2}\exp\left(-\sum_{j=0}^d\frac{x_j^2}{2}\right)h_\gamma(x),
\eeqsn
and the Hermite polynomials $h_\gamma$ are given by
\beqsn
h_\gamma(x):=(-1)^{|\gamma|}\exp\left(\sum_{j=0}^dx_j^2\right)
\partial^\gamma\exp\left(-\sum_{j=0}^dx_j^2\right),\qquad x\in\R^d.
\eeqsn

Here we consider a matrix $A^*$ 
of K\"othe type with positive entries as in Section~\ref{sec3}
for $\Lambda=\N_0^d$, defined by
\beqs
\label{ank}
a_{\gamma,k}:=e^{M(k|\gamma|^{1/2})}, \qquad \gamma\in\N_0^d, \, k\in\N,
\eeqs
where $M(t)$ is the associated function defined by \eqref{Mt}.
We characterize when $\Sch_{(M_p)}(\R^d)$ is nuclear with
Theorem~3.5 of \cite{P}, that we state here in our setting, for the convenience of the reader. In what follows we denote $\lambda^1:=\tilde\lambda^1(A^*)$ and $\lambda^\infty:=\tilde\lambda^\infty(A^*)$.
\begin{Th}
\label{th35P}
Assume that the inclusion $j:\,\lambda^1\to\lambda^\infty$ has dense image.
Let $E$ be a locally convex space such that we have a commutative diagram of
continuous linear operators of the form

\vspace{2mm}
\hspace*{35mm}\begin{tikzpicture}
\draw[->] (-1.5,4)--(0.6,4);
\draw[->] (1,3.7)--(1,1.5);
\draw[->] (-1.8,3.8)--(0.7,1.5);
\node at (-2,4.1) {$\lambda^1$};
\node at (1,1.2) {$\lambda^\infty$};
\node at (-0.3,4.4) {$T$};
\node at (1.3,2.7) {$S$};
\node at (-1,2.3) {$j$};
\node at (1,4.1) {$E$};
\end{tikzpicture}

with $S$ injective or $T$ with dense image. Then $\lambda^1$ is nuclear if and only if $E$
is nuclear.
\end{Th}

We can now prove the following:
\begin{Prop}
\label{thSMp}
Let $(M_p)_p$ be a sequence satisfying $M_p^{1/p}\to+\infty$ as $p\to+\infty$, condition \eqref{12L} and $(M1)$.
Then $\Sch_{(M_p)}(\R^d)$ is nuclear if and only if the associated function $M(t)$ satisfies
\beqs
\label{3BJO-Rodino}
\exists H>1\ \mbox{s.t.}\ M(t)+\log t\leq M(Ht)+H,\quad\forall t>0.
\eeqs
\end{Prop}

\begin{proof}
We shall use Theorem~\ref{th35P} with $E=\Sch_{(M_p)}(\R^d)$.
We observe that $\lambda^1\subseteq\lambda^\infty$ and denote by $j$ the inclusion
\beqsn
j:\ \lambda^1\longrightarrow \lambda^\infty.
\eeqsn
Let us consider the linear map
\beqsn
S:\ \Sch_{(M_p)}(\R^d)&&\longrightarrow\lambda^\infty\\
f&&\longmapsto (c_\gamma)_{\gamma\in\N_0^d}:=(\xi_\gamma(f))_{\gamma\in\N_0^d},
\eeqsn
where 
\beqsn
\xi_\gamma(f)=\int_{\R^d}f(x)H_\gamma(x)dx
\eeqsn
are the Hermite coefficients of $f$, and then the linear map
\beqsn
T:\ \lambda^1&&\longrightarrow\Sch_{(M_p)}(\R^d)\\
(c_\gamma)_{\gamma\in\N_0^d}&&\longmapsto \sum_{\gamma\in\N_0^d}
c_\gamma H_\gamma(x).
\eeqsn

In Theorem~3.4 of \cite{L} it was proved that condition \eqref{12L} implies that $S$
and $T$ are continuous.
%
%
Note also that the diagram in Theorem~\ref{th35P} commutes by the uniqueness
of the coefficients with respect to the Schauder basis
$(H_\gamma)_{\gamma\in\N_0^{d}}$.


Let us prove that $j$ has dense image. By conditions
$M_p^{1/p}\to+\infty$ and
(M1), and  by \cite[Lemma~3.2]{P}, we have
\beqsn
\lim_{t\to+\infty}e^{M(t/h)-M(t/h')}=0,\qquad\mbox{if}\ h>h'>0.
\eeqsn
Therefore, for every $k\in\N$ there exists $m\in\N$, $m>k$, such that
\beqsn
\lim_{|\gamma|\to+\infty}a_{\gamma,k}a_{\gamma,m}^{-1}=
\lim_{|\gamma|\to+\infty}e^{M(k|\gamma|^{1/2})-M(m|\gamma|^{1/2})}=0,
\eeqsn
and hence $\lambda^\infty=\tilde c_0(A^*)$, by the same arguments we used to
prove that \eqref{star} implies $\tilde\lambda^\infty(A)=\tilde c_0(A)$ in Section~\ref{sec3}.
Then $j(\lambda^1)$ is dense in $\lambda^\infty= \tilde c_0(A^*)$ because
\beqsn
{\tilde c}_{00}(A^*):=\{(c_\gamma)_{\gamma\in\N_0^d}\in \tilde c_0(A^*):\ c_\gamma=0\ \mbox{except that for a finite number of
indexes}\}
\eeqsn
is dense in $\tilde c_0(A^*)$ and is contained in $\lambda^1$.

Moreover, $S$ is injective. Hence, by Theorem~\ref{th35P},
$E=\Sch_{(M_p)}(\R^d)$ is nuclear if and only if $\lambda^1$ is
nuclear. By Theorem~\ref{prop2816MV}, the sequence space
$\lambda^1$ is
nuclear if and only if
\beqs
\label{condseries}
\forall k\in\N\ \exists m\in\N, m\geq k\ \mbox{s.t.}\quad
\sum_{\gamma\in\N_0^d}e^{M(k|\gamma|^{1/2})-M(m|\gamma|^{1/2})}<+\infty.
\eeqs
The
series in \eqref{condseries} converges if and only if 
\beqs
\label{48}
M(t)+N\log t\leq M(H^Nt)+C_{N,H},\qquad\forall N\in\N,
\eeqs
for some $C_{N,H}>0$ and $N>2d$ (see the proof of \cite[Thm. 1]{BJO-Rodino}). This gives the conclusion since \eqref{48} is equivalent to \eqref{3BJO-Rodino} (see again the proof of \cite[Thm. 1]{BJO-Rodino}).
\end{proof}

\begin{Th}
\label{corSMp}
Let $(M_p)_p$ be a sequence satisfying $M_p^{1/p}\to+\infty$ as $p\to+\infty$, condition \eqref{12L} and $(M1)$.
Then $\Sch_{(M_p)}(\R^d)$ is nuclear if and only if
 $(M2)'$ holds.
\end{Th}

\begin{proof}
It follows from Proposition~\ref{thSMp} because, under condition (M1),
condition (M2)$'$ is equivalent to condition \eqref{3BJO-Rodino}
(see \cite[Rem.~1]{BJO-Rodino}).
\end{proof}

If $(M2)'$ is satisfied then $\Sch_{(M_p)}(\R^d)$ can be equivalently defined with
 $L^\infty$ norms as in \eqref{SMp-infty} (see \cite[Remark 2.1]{L}) and hence $S_{(M_p)}(\R^d)$ is nuclear (cf. \cite[Corollary 1]{BJO-Rodino}), but we cannot derive a characterization in terms of $(M2)'$ from the results of Langenbruch~\cite{L}.

\vskip\baselineskip
{\bf Acknowledgments.} 
The first three authors were partially supported by  the Project FFABR 2017 (MIUR),
and by the Projects FIR 2018 and FAR 2018 (University of Ferrara).
The first and third  authors are members of the Gruppo Nazionale per l'Analisi
Matematica, la Probabilit\`a e le loro Applicazioni (GNAMPA) of the Istituto Nazionale di Alta
Matematica (INdAM). The research of the second author was partially supported by the project MTM2016-76647-P and the grant BEST/2019/172 from Generalitat Valenciana. The fourth author is supporded by FWF-project J 3948-N35.

\end{document}